\documentclass[a4paper,11pt]{amsart}
\usepackage[margin= 2.5 cm,bottom=19mm]{geometry}
\usepackage{thm-restate}
\usepackage{graphicx}
\usepackage{geometry}
\usepackage{amsthm}
\usepackage{amssymb}
\usepackage{amsmath}
\usepackage{comment}
\usepackage{hyperref}
\usepackage{xcolor}
\usepackage{enumerate}
\usepackage{enumitem}
\usepackage{listings}
\usepackage{complexity}
\usepackage{blkarray}
\usepackage{lmodern}
\usepackage[english]{babel}
\usepackage{accents}
\usepackage{float}
\usepackage{tikz}
\usetikzlibrary{decorations.pathmorphing}

\tikzset{snake it/.style={decorate, decoration=snake}}

\usetikzlibrary{calc}
\usetikzlibrary{decorations.pathreplacing}
\usetikzlibrary{positioning,patterns}
\usetikzlibrary{arrows,shapes,positioning}
\usetikzlibrary{decorations.markings}

\tikzstyle{edge}=[very thick]
\definecolor{bostonuniversityred}{rgb}{0.8, 0.0, 0.0}
\definecolor{arsenic}{rgb}{0.23, 0.27, 0.29}
\tikzstyle{diredge}=[postaction={decorate,decoration={markings,
		mark=at position .95 with {\arrow[scale = 1]{stealth};}}}]

\newcommand{\fitellipsis}[2] 
{\draw [fill=gray]let \p1=(#1), \p2=(#2), \n1={atan2(\y2-\y1,\x2-\x1)}, \n2={veclen(\y2-\y1,\x2-\x1)}
    in ($ (\p1)!0.5!(\p2) $) ellipse [ x radius=\n2/2+0cm, y radius=0.4cm, rotate=\n1];
}

\usepackage[font=small,labelfont=bf]{caption}
\usepackage[font=small,labelfont=normalfont,labelformat=simple]{subcaption}
\graphicspath{{figs/}}

\newtheorem{theorem}{Theorem}[section]
\newtheorem{lemma}[theorem]{Lemma}

\newtheorem{proposition}[theorem]{Proposition}

\newtheorem{remark}[theorem]{Remark}
\newtheorem{problem}[theorem]{Problem}

\theoremstyle{definition}
\newtheorem{definition}[theorem]{Definition}

\author[Das, Dragani\'{c}, Steiner]{Shagnik Das  \and Nemanja Dragani\'{c}  \and Raphael Steiner }
\address[Das]{Department of Mathematics, National Taiwan University, Taiwan}
\address[Dragani\'{c}]{Department of Mathematics, ETH Z\"{u}rich, Switzerland}
\address[Steiner]{Institute of Theoretical Computer Science, ETH Z\"{u}rich, Switzerland}
\email{\tt shagnik@ntu.edu.tw}
\email{\tt nemanja.draganic@math.ethz.ch}
\email{\tt raphaelmario.steiner@inf.ethz.ch}

\thanks{The first author was supported by the Deutsche Forschungsgemeinschaft (DFG) project 415310276. The second author was supported in part by SNSF Grant 200021\_196965. The third author was supported by an ETH Z\"{u}rich Postdoctoral Fellowship.
}

\date{\today}

\title{Tight bounds for divisible subdivisions}
\begin{document} 
\maketitle

\begin{abstract}
Alon and Krivelevich proved that for every $n$-vertex subcubic graph $H$ and every integer $q \ge 2$ there exists a (smallest) integer $f=f(H,q)$ such that every $K_f$-minor contains a subdivision of $H$ in which the length of every subdivision-path is divisible by $q$. Improving their superexponential bound, we show that $f(H,q) \le \frac{21}{2}qn+8n+14q$, which is optimal up to a constant multiplicative factor.
\end{abstract}

\section{Introduction}
Enjoying a long tradition in graph theory, the typical extremal question asks which conditions we can impose on a graph to force its containment of a given subgraph. Famous examples of this class of problems are Tur\'{a}n's theorem~\cite{Tur41}, describing the smallest average degree that guarantees the existence of a complete subgraph of specified size, and Dirac's Theorem~\cite{Dir52}, which gives a sharp minimum degree threshold for the existence of a Hamiltonian cycle. In this paper we shall be concerned with the existence of subdivisions of a fixed graph, and well-known results in this direction are due to Bollob\'{a}s and Thomason~\cite{bolthom} and Koml\'{o}s and Szemer\'{e}di~\cite{komszem}, who showed that any graph of average degree $\Theta(t^2)$ contains a subdivision of $K_t$. 

While results of this kind are of fundamental importance in extremal and structural graph theory, they share the shortcoming that they require the host graphs to be reasonably dense, and do not, for instance, yield anything for graphs with bounded maximum degree. These very sparse graphs arise naturally in several applications, and it is therefore of interest to study other structural conditions, different from the average or minimum degree, that guarantee the existence of desired subgraphs (necessarily themselves of small maximum degree). We follow this line of research by solving a problem introduced by Alon and Krivelevich~\cite{alon} regarding the existence of so-called \emph{divisible subdivisions} in graphs containing a large clique minor.

Before stating our main result, let us describe the necessary definitions and background. Throughout this paper, a \emph{$K_f$-minor} is  defined as a graph $G$ whose vertex set is partitioned into $f$ disjoint non-empty sets $X_1,\ldots,X_f$, such that for every $i \in [f]$, the induced subgraph $G[X_i]$ is connected, and for every $i \neq j \in [f]$, there exists at least one edge in $G$ with endpoints in $X_i$ and $X_j$. The sets $X_1,\ldots,X_f$ are referred to as \emph{supernodes} or \emph{branch sets} of the $K_f$-minor $G$. 

Given a graph $H$, a \emph{subdivision} of $H$ is any graph $H'$ obtained from $H$ by replacing its edges with internally vertex-disjoint paths connecting the original endpoints of the edges in $H$. A vertex in $H'$ corresponding to an original vertex of $H$ is called a \emph{branch vertex} of $H'$, while all remaining vertices are called \emph{subdivision vertices}. The paths in $H'$ replacing the edges of $H$ are called \emph{subdivision paths}.
A subdivision $H'$ of $H$ is called \emph{$q$-divisible} if all its subdivision paths are of length divisible by $q$. 

Subdivisions and minors of graphs with constraints on the lengths of paths have received significant attention in the literature. For example, Alon, Krivelevich and Sudakov~\cite{alon2} showed that every $n$-vertex graph of average degree $\varepsilon n$, for fixed $\varepsilon>0$, contains a subdivision of $K_k$ in which every subdivision-path has length $2$ (thus, a $2$-divisible subdivision), where $k=\Omega(\sqrt{n})$. For more results of this nature, we refer the reader to ~\cite{bollobas,geelen,drag1,drag2,gir,kaw,kaw1,kaw2,kaw3,kaw4,sewell,thomassen2}.

Another standout result is due to Thomassen~\cite{thomassen}, who gave general sufficient conditions for finding subdivisions with modular constraints on the lengths of the subdivision paths. He proved that, given any graph $H$ and, for every edge $e \in E(H)$, an assignment of two natural numbers $d(e)$ and $k(e)$, there exists an integer $c$ (depending only on $H$ and the sequences $d(e), k(e)$) such that every graph of chromatic number at least $c$ contains a subdivision of $H$ in which each edge $e \in E(H)$ is replaced by a subdivision path whose length is congruent to $d(e)$ modulo $k(e)$. Furthermore, if for no edge $e \in E(H)$ the number $d(e)$ is odd while $k(e)$ is even, then there exists an integer $c'$ such that for every graph of minimum degree at least $c'$ the same conclusion holds. The latter result in particular shows that for every graph $H$ and every integer $q$, any graph of sufficiently large minimum degree (in terms of $H$ and $q$) contains a $q$-divisible $H$-subdivision.

As alluded to earlier, the results mentioned above, and other results on parity-constrained subdivisions in the literature, only apply to dense graphs. Thus, if the average degree of the host graph is small (just slightly above $2$, say), then no sufficient conditions for divisible subdivisions were known. Alon and Krivelevich~\cite{alon} filled this gap by providing a much more general sufficient condition in the case that $H$ is subcubic. 

\begin{theorem} \label{thm:alonkrivelmainresult}
For every graph $H$ with $\Delta(H) \le 3$ and every integer $q \ge 2$ there exists a (smallest) integer $f=f(H,q) \ge 1$ such that every $K_f$-minor contains a $q$-divisible subdivision of $H$ as a subgraph.
\end{theorem}

The main advantage of the result of Alon and Krivelevich relies in the fact that it is possible to find (arbitrarily) large complete minors even in classes of graphs with bounded maximum degree. For example, it was shown by Kawarabayashi and Reed~\cite{kawreed} that every graph without sublinear separators contains a large complete minor; in particular it is known (c.f.~Krivelevich~\cite{krivel}) that this property holds for essentially all graphs (even those of bounded degree) with sufficiently good expansion properties. Another interesting class of graphs for which the result applies are graphs of minimum degree at least $3$ and large girth (see the results in~\cite{thomassengirth,diestelgirth,kuhngirth}).  Hence, Theorem~\ref{thm:alonkrivelmainresult} guarantees the existence of $q$-divisible $H$-subdivisions in these sparse graph classes.

Theorem~\ref{thm:alonkrivelmainresult} is qualitatively optimal in the sense that for every $q \ge 2$ there exist complete $K_f$-minors for arbitrarily large $f$ with maximum degree $3$ and such that every path between a pair of vertices of degree $3$ has length divisible by $q$. Such a minor would not contain a subdivision of a graph with maximum degree at least $4$, nor would it contain a subdivision of a cubic graph with a subdivision path of non-zero length modulo $q$.

However, the result proved in~\cite{alon} was quantitatively far from optimal: following the proof of Theorem~\ref{thm:alonkrivelmainresult} in~\cite{alon}, it gives at best an upper bound on $f(H,q)$ which is of magnitude $(q^2n)^{q^3n}$, where $n=v(H)$. In contrast, the best lower bound on $f(H,q)$ for subcubic graphs $H$ on $n$ vertices and $m$ edges we are aware of is $m(q-1) + n$, obtained by considering the complete graph of order $m(q-1)+n-1$ (which is too small to host a $q$-divisible subdivision of~$H$). 

Consequently, Alon and Krivelevich posed the problem of improving the bound on $f(H,q)$. In this paper we determine the value of $f(H,q)$ for all possible choices of $H$ and $q$ up to a constant multiplicative error.

\begin{theorem}\label{thm:mainresult}
Let $H$ be an $n$-vertex graph with $e(H) = m$ and $\Delta(H)  \le 3$. Then, for every integer $q \ge 2$, it holds that
$$m(q-1)+n \le f(H,q) \le 7mq+8n+14q,$$ and hence $f(H,q)=\Theta(mq+n)$.
\end{theorem}

In the special case where $H$ is a cycle, Alon and Krivelevich~\cite{alon} proved that for some constant $C$, every $K_f$-minor with $f \ge C q \log q$ contains a cycle of length divisible by $q$. The correct order of magnitude in this case was determined independently by M\'{e}sz\'{a}ros and the third author~\cite{tamas} and by Arsovski (personal communication), who showed that $f \ge C q$ is sufficient. Theorem~\ref{thm:mainresult} thus completes the picture by extending these sharp bounds to all subcubic graphs.

In fact, we obtain Theorem~\ref{thm:mainresult} as a special case of a more general result, which naturally generalises the setting of divisible subdivisions to that of Abelian groups, as also mentioned by Alon and Krivelevich. Given an abelian group $(A,+)$, we call a graph \emph{$A$-weighted} if each edge of the graph is equipped with a weight $a \in A$. This setting is studied in a subfield of algebraic Ramsey theory known as \emph{zero-sum theory}, and we refer the reader to the survey of Caro~\cite{caro} for an overview of the area. One important graph-theoretic parameter studied in the area is the \emph{zero-sum Ramsey number}, which measures the smallest size of a weighted complete graph in which one is guaranteed to find a desired subgraph of total weight zero. We refer to~\cite{alon3,bialo,caro1,caro2,caro3,caro4,caro5,caro6,caro7} for examples of results on this parameter. 

Here, instead of a $q$-divisible subdivision, we aim for finding a subdivision of a fixed graph $H$ such that the sum of the weights along any subdivision path equals $0 \in A$, and call such a subdivision an \emph{$A$-divisible $H$-subdivision}.

The bound we will obtain for this problem depends on the parameter of an abelian group defined by
$$\sigma(A) = \max_{B \le A}\left\{\frac{|\{a \in A : 2a \in B\}|}{|B|}\right\}.$$
For example, observe that for every integer $q \ge 2$, we have $\sigma(\mathbb{Z}_q)=1$ if $q$ is odd and $\sigma(\mathbb{Z}_q)=2$ if $q$ is even. With this notation in place, we can now state our result.

\begin{restatable}{theorem}{abelian}\label{thm:groupresult}
For every subcubic graph $H$ with $n$ vertices and $m$ edges and for every finite abelian group $(A,+)$, it holds that every $A$-weighted $K_f$-minor with $$f\geq 7m|A|+ 4 n \sigma(A) + 14 |A|$$
contains an $A$-divisible $H$-subdivision. 
\end{restatable}

We can deduce Theorem~\ref{thm:mainresult} directly from Theorem~\ref{thm:groupresult} as follows: given a $K_f$-minor $G$ with $f \ge 7mq+8n+14q$, let $A=(\mathbb{Z}_q,+)$ and assign weight $1 \in A$ to each edge of $G$. We can now apply Theorem~\ref{thm:groupresult}, observing that $f \ge 7m |A| + 4n \sigma(A) + 14 |A|$. In this labelling of $G$, a $\mathbb{Z}_q$-divisible $H$-subdivision is precisely a $q$-divisible $H$-subdivision, and hence we recover the conclusion of Theorem~\ref{thm:mainresult}.

\medskip

\paragraph*{\textbf{Notation and Organisation.}} We use standard notation throughout, but highlight some key terminology here to avoid any confusion.

Given a graph $G$ and a finite abelian group $(A,+)$, an \emph{$A$-weighting} of $G$ is defined to be an assignment $w:E(G) \rightarrow A$ of elements from $A$ to the edges of $G$. Given an $A$-weighting $w$ of a graph $G$ and a subgraph $H$ of $G$, we denote by $w(H)=\sum_{e \in E(H)}{w(e)}$ the total weight of $H$ in $G$, where the summation is in the abelian group. An $A$-weighted $K_f$-minor for some $f \ge 1$ is simply a $K_f$-minor equipped with an $A$-weighting (which we, if not defined otherwise, always denote by $w$). 

If $G$ is an $A$-weighted graph for some abelian group $(A,+)$, we say that a subdivision of $H$ contained as a subgraph in $G$ is an \emph{$A$-divisible subdivision} if the sum of all edge weights along any subdivision path in the $H$-subdivision equals $0$. 

Given an abelian group $(A,+)$ and subsets $A_1, A_2, \ldots, A_k \subseteq A$, we denote by $A_1+\cdots+A_k=\{a_1+\cdots+a_k : a_i \in A_i, i \in [k] \} \subseteq A$ the sumset of $A_1,\ldots,A_k$. The sum of an empty list of subsets of $A$ is defined to be equal to $\{0\}$. Given a subset $S \subseteq A$, we denote by $\langle S \rangle$ the set of elements in the subgroup generated by the elements in $S$. 

\medskip

The remainder of this paper is laid out as follows. In the following section, we present some preliminary definitions and results that will be used in our proof. In Section~\ref{sec:proof}, we prove our main result, Theorem~\ref{thm:groupresult}. Finally, in Section~\ref{sec:conclusion}, we provide some concluding remarks and open questions.

\section{Preliminaries}
In this section we shall introduce the notion of connectors, which play a pivotal role in our proof. We shall define them and prove some initial results, leading up to Proposition~\ref{prop:connectors}, which concerns the existence of connectors in $A$-weighted $K_f$ minors (in what follows, $A$ is an arbitrary finite abelian group and $f \ge 4$ is an integer).

\subsection{Reduced minors and connectors}

Since our main result is about finding certain subgraphs in $A$-weighted complete minors, it will be helpful for us to restrict our attention to minors which are in some sense minimal. This property is captured in the following definition.
\begin{definition}

Let $G$ be an $A$-weighted $K_f$-minor. We say that $G$ is a \emph{reduced minor} if the following hold:
\begin{itemize}
    \item Each supernode induces a tree in $G$;
    \item Every leaf in each tree induced by a supernode is adjacent to a vertex in another supernode;
    \item There is exactly one edge between any two supernodes;
    \item $\delta(G) \ge 3$.
\end{itemize}
\end{definition}

\begin{remark}\label{rem:minimal}
Let $G$ be an $A$-weighted $K_f$-minor, where $f\geq 4$. We can get a reduced minor $G'$ from $G$ in the following natural way.
We start by, for each supernode $N$, replacing $G[N]$ by one of its spanning trees. Next, between every pair of supernodes, we remove excess edges until exactly one connecting edge remains. We continue by, for every supernode $N$, successively removing from the tree $G[N]$ leaves that do not have any neighbours outside $N$. Notice that after each of those operations, our graph is a $K_f$-minor with minimum degree at least $2$.

Finally, for every vertex $v$ with only two neighbours $u_1$ and $u_2$, we delete $v$ and introduce the edge $\{u_1, u_2\}$ with weight $w(\{u_1, u_2\}) = w(\{u_1,v\}) + w(\{v,u_2\})$. Observe that $u_1$ and $u_2$ cannot previously have been adjacent, since if they were, then $v, u_1$ and $u_2$ would form a triangle. However, since $f \ge 4$, $v$ is of too low degree to be its own supernode, and therefore must be in the same supernode as at least one of $u_1$ and $u_2$. Then, if $u_1$ and $u_2$ are in the same supernode $N$, we would have a cycle in $G[N]$, while otherwise there would be two edges between their supernodes. 

Hence, after completing these steps, we obtain an $A$-weighted $K_f$-minor $G'$, which we say is a \emph{reduced graph} of $G$. Crucially, note that every path in $G'$ is obtained as a contraction of a path in $G$ with the same endpoints and the same weight.

The reduction of $A$-weighted clique minors described above naturally gives rise to a containment relation as follows: Given two numbers $f_1, f_2 \in \mathbb{N}$ such that $f_1 \le f_2$, an $A$-weighted $K_{f_1}$-minor $G_1$ and an $A$-weighted $K_{f_2}$-minor $G_2$, we say that $G_1$ is a reduced sub-minor of $G_2$, in symbols, $G_1 \preceq_A G_2$, if $G_1$ is a reduced minor obtained from $G_2$ by first deleting all vertices in a subset of its supernodes, and then applying reduction operations as described above to the remaining complete minor.
\end{remark}

Pause to note that $\preceq_A$ forms a transitive relation on the set of $A$-weighted clique minors, and, just as above, if $G_1 \preceq_A G_2$ for two $A$-weighted clique-minors, then every path in $G_1$ corresponds to a contraction of a path in $G_2$ with the same endpoints and the same weight. In particular, if $G_1$ contains an $A$-divisible subdivision of a graph $H$, then $G_2$ contains such a subdivision as well, which even uses the same set of branch vertices.

When working with a reduced $K_f$-minor, it is often convenient to view the graph at the level of its supernodes. However, when we then try to embed a subdivision $H'$ of a cubic graph $H$, we need to identify individual vertices within the supernodes to act as the branch vertices of $H'$. The following proposition allows us to do so.

\begin{proposition}\label{lem:central}
Let $T$ be a tree, and let $v_1$, $v_2$ and $v_3$ be (not necessarily distinct) vertices in $T$. Then there exists a vertex $\bar{v}$ that is connected by internally vertex-disjoint paths (possibly of length zero) to $v_1$, $v_2$ and $v_3$. We call $\bar{v}$ the central vertex in $T$ with respect to $v_1$, $v_2$ and $v_3$.
\end{proposition}

\begin{proof}
Let $P$ be the unique path between $v_1$ and $v_2$. If $v_3$ is on $P$, then we set $\bar{v}=v_3$. If not, consider the path from $v_3$ to $v_2$, and let $\bar{v}$ be the first vertex on this path that lies on $P$.
\end{proof}

We conclude this subsection by introducing connectors, which are central to our proof. Roughly speaking, these are subgraphs of weighted graphs that contain many paths of different weights between a specified pair of vertices.

\begin{definition}\label{def:connectors}
Given an $A$-weighted graph $G$, a \emph{connector} is a subgraph of $G$ consisting of the union of disjoint cycles $C_1,\ldots,C_\ell$ together with paths $P_0,P_1,\ldots, P_{\ell}$ such that all paths are mutually disjoint and internally vertex disjoint from the cycles, and such that the last vertex of $P_i$ is in $C_{i+1}$ for all $0\leq i\leq \ell-1$, and the first vertex of $P_i$ is in $C_i$ for all $i\in[\ell]$. For all $i \in [\ell]$, let $x_i$ and $y_i$ be the weights of the two paths contained in $C_i$ that connect $P_{i-1}$ to $P_i$.

For a subset $S\subseteq A$ we say that a connector is an \emph{$S$-connector} if $S \subseteq \sum_{i\in [\ell]}\{0,x_i-y_i\}$. We will refer to the first vertex of $P_0$ and the last vertex of $P_\ell$ as the first and last vertex of the connector, respectively (see Figure \ref{fig:connector}), and we call them the \emph{endpoints} of the connector. 
\end{definition}

\begin{figure}[ht]
    \centering
    \begin{tikzpicture}[scale=1.2,main node/.style={circle,draw,color=black,fill=black,inner sep=0pt,minimum width=3pt}]
        
	    \node[main node] (a) at (1,0) [label=below:$u$]{};
	    \node[main node] (a1) at (2,0){};
	    \node[main node] (a2) at (3.5,0){};
	    \node[main node] (a3) at (4.5,0){};
	    \node[main node] (a4) at (6,0){};
	    \node[main node] (a5) at (7,0){};

        \node[] (P0) at(1.5,0.2) {$P_0$};
        \node[] (P0) at(2.75,0) {$C_1$};
        \node[] (P0) at(2.75,0.9) {$x_1$};
        \node[] (P0) at(2.75,-0.9) {$y_1$};
        
        \node[] (P0) at(4,0.2) {$P_1$};
        \node[] (P0) at(5.25,0) {$C_2$};
        \node[] (P0) at(5.25,0.9) {$x_2$};
        \node[] (P0) at(5.25,-0.9) {$y_2$};
        
        \node[] (P0) at(6.5,0.2) {$P_2$};
         
        \node[] (P0) at(10,0.2) {$P_{\ell-1}$};
        \node[] (P0) at(11.25,0) {$C_\ell$};
        \node[] (P0) at(11.25,0.9) {$x_\ell$};
        \node[] (P0) at(11.25,-0.9) {$y_\ell$};
        
        \node[] (P0) at(12.5,0.2) {$P_\ell$};

        \draw[thick] (a)--(a1);
        \draw[thick](2.75,0) circle (0.75);
          \draw[thick] (a2)--(a3);
          \draw[thick](5.25,0) circle (0.75);
          \draw[thick] (a4)--(a5);
	    \draw[loosely dotted,thick] (7.5,0)--(9,0);
	    
	    \node[main node] (a) at (9.5,0) {};
	    \node[main node] (a1) at (10.5,0){};
	    \node[main node] (a2) at (12,0){};
	    \node[main node] (a3) at (13,0)[label=below:$v$]{};
	     \draw[thick] (a)--(a1);
        \draw[thick](11.25,0) circle (0.75);
          \draw[thick] (a2)--(a3);
	    

    \end{tikzpicture}
    \caption{Illustration of a connector with endpoints $u$ and $v$. For each $i\in [\ell]$, the labels $x_i$ and $y_i$ represent the weights of the two paths in $C_i$ which connect $P_{i-1}$ to $P_i$.}
    \label{fig:connector}
\end{figure}
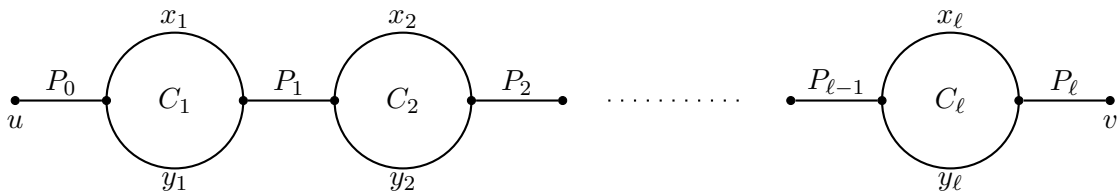

The \emph{base path} of a connector is the path $P$ between the endpoints such that, for each $i$, $P\cap C_i$ is the path of weight $y_i$. By \emph{switching} the path used in $C_i$ to the one which is of weight $x_i$, observe that the weight of $P$ changes by $x_i-y_i$. Hence, by definition of a $S$-connector, by doing the appropriate switches, for every $s\in S$ we can get a path $Q$ of weight $w(Q)=w(P)+s$ between the endpoints of the connector.

\subsection{Permissible cycles and $B$-restricted minors}
The main goal of this section is to prove Proposition~\ref{prop:connectors}, which concerns the existence of connectors and plays an important role in the proof of Theorem~\ref{thm:groupresult}. In order to lay the groundwork for this result, we next define special types of paths and cycles in complete minors. 
\begin{definition}
Let $G$ be a reduced $K_f$-minor.
\begin{itemize}
    \item We say that a path $P$ in $G$ is \emph{permissible} if for every supernode $N$ we have that $G[N]\cap P$ is either empty or a path.
    \item We say that a cycle $C$ in $G$ is \emph{permissible} if for every supernode $N$ we have that $G[N]\cap C$ is either empty or a path, except for at most one supernode $N_1$ for which $G[N_1]\cap C$ can be the union of two vertex-disjoint paths.
\end{itemize}
\begin{remark}\label{rem:permissible}
Notice that if a permissible cycle $C$ has a supernode $N_1$ for which $G[N_1]\cap C$ is the union of two disjoint paths $P_1$ and $P_2$, then $C$ is incident to at least five supernodes. Indeed, observe that the four edges in $C$ which are incident to the endpoints of $P_1$ and $P_2$, and which are not in $N_1$, have their other endpoint in mutually distinct supernodes. This is because $G$ is a reduced minor, so there is exactly one edge between two supernodes. Hence, $C$ meets at least four other supernodes besides $N_1$.
\end{remark}
\end{definition}
We will also need the following definition before stating the results of this subsection.

\begin{definition}
Let $G$ be an $A$-weighted $K_f$-minor and let $B\leq A$ be a subgroup. Then we say that $G$ is \emph{$B$-restricted} if the following hold:
\begin{itemize}
    \item Every permissible cycle $C$ in $G$ has weight $w(C)\in B$.
    \item Every edge $e$ has weight $w(e)$ such that $2w(e)\in B$.
\end{itemize}
\end{definition}

\noindent Note that every $A$-weighted $K_f$-minor is trivially $A$-restricted, and so this definition is only of interest when $B$ is a proper subgroup of $A$. The following lemma states that in order to show a weighted minor is $B$-restricted, it suffices to check permissible cycles incident to few supernodes.

\begin{lemma}\label{lem:even paths}
Let $A$ be an abelian group, let $f\geq 5$ and let $B\leq A$. Let $G$ be a reduced $A$-weighted $K_f$-minor such that every permissible cycle $C$ whose vertices are contained in at most five supernodes satisfies $w(C)\in B$. Then $G$ is $B$-restricted.
\end{lemma}
\begin{proof}
Suppose $B\neq A$ as otherwise the claim holds trivially. Let us first show that $2w(e)\in B$ for every edge $e$ in $G$. Denote by $u,v$ the endpoints of an edge $e$, and let $N_1$ and $N_2$ be the supernodes such that $u \in N_1$ and $v \in N_2$. 

First consider the case that $N_1=N_2$. Since $u$ and $v$ are of degree at least three (since $G$ is a reduced minor), let $x_1,x_2$ be two of the other neighbours of $u$, and $x_3,x_4$ two other neighbours of $v$. For each $i\in[4]$ do the following. If $x_i$ is already in another supernode $X_i\neq N_1$, let $P_i$ be the path consisting of one vertex $x_i$, and let $s_i=x_i$.
Otherwise, let $t_i$ be a leaf in the maximal subtree of $G[N_1]$ which contains $x_i$, but which does not contain $u$ and $v$, and let $s_i$ be a neighbor of $t_i$ in another supernode $X_i$ (by the definition of a reduced minor, every leaf has such a neighbour). Now, let $P_i$ be the path obtained by concatenating the unique path between $x_i$ and $t_i$ in $G[N_1]$ with the edge $\{t_i,s_i\}$. Next, let $Q_1$ be the unique path between $s_1$ and $s_3$ in the tree $G[X_1\cup X_3]$, and similarly $Q_2$ the path between $s_2$ and $s_4$ in $G[X_2\cup X_4]$ (see Figure \ref{fig:edgeweight}). Note that by construction we have that the cycle $C_1$ formed by the paths $P_1-Q_1-P_3- x_3-v-u-x_1$ is a permissible cycle incident with three supernodes, as is the cycle $C_2$ consisting of $P_2-Q_2-P_4-x_4-v-u-x_2$, so we have that $w(C_1),w(C_2)\in B$. Also note that the cycle $C$ with edge-set $(E(C_1)\cup E(C_2))\setminus\{e\}$ is permissible and incident to five supernodes, hence we conclude $2w(e)=w(C_1)+w(C_2)-w(C)\in B$.
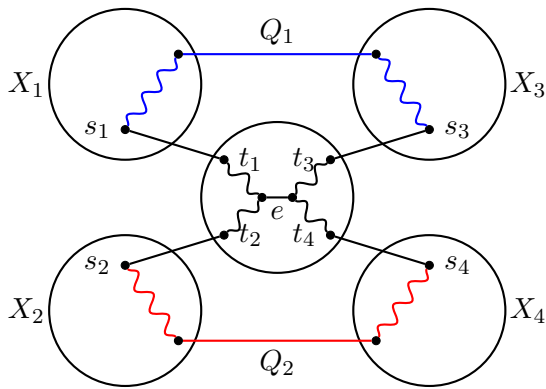
\begin{figure}[ht]
    \centering
    \begin{tikzpicture}[scale=1,main node/.style={circle,draw,color=black,fill=black,inner sep=0pt,minimum width=3pt}]


	    \draw[thick](0,0) circle (1);
	        \node[main node] (a) at (-0.2,0) {};
	        \node[main node] (b) at (0.2,0) {};
	        \draw[] [thick] (a) -- (b) ;
	        \node[main node] (x1) at (-0.7,0.5)[label=right:$t_1$]{};
	        \node[main node] (x3) at (0.7,0.5)[label=left:$t_3$]{};
	        \node[main node] (x2) at (-0.7,-0.5)[label=right:$t_2$]{};
	        \node[main node] (x4) at (0.7,-0.5)[label=left:$t_4$]{};
	        \path [draw=black,snake it, thick] (a) -- (x1);
	        \path [draw=black,snake it, thick] (a) -- (x2);
	        \path [draw=black,snake it, thick] (b) -- (x3);
	        \path [draw=black, snake it, thick] (b) -- (x4);
	    \draw[thick](2,1.5) circle (1);
	    \node[main node] (a3) at (2,0.9) [label=right:$s_3$]{};
	    \node[main node] (b3) at (1.3,1.9){};
	    \path [draw=blue, snake it, thick] (a3) -- (b3);
	    \draw[] [thick] (x3) -- (a3) ;

        \draw[thick](-2,1.5) circle (1);
        \node[main node] (a1) at (-2,0.9)[label=left:$s_1$]{};
	    \node[main node] (b1) at (-1.3,1.9){};
	    \path [draw=blue, snake it, thick] (a1) -- (b1);
	    \draw[] [thick] (x1) -- (a1) ;
         
        \draw[blue,thick] (b1)--(b3);
        
        \draw[thick](2,-1.5) circle (1);
        \node[main node] (a4) at (2,-0.9)[label=right:$s_4$]{};
	    \node[main node] (b4) at (1.3,-1.9){};
	    \path [draw=red, snake it, thick] (a4) -- (b4);
	    \draw[] [thick] (x4) -- (a4) ;
        
        \draw[thick](-2,-1.5) circle (1);
        \node[main node] (a2) at (-2,-0.9)[label=left:$s_2$]{};{};
	    \node[main node] (b2) at (-1.3,-1.9){};
	    \path [draw=red, snake it, thick] (a2) -- (b2);
	    \draw[] [thick] (x2) -- (a2) ;
	    
	    \draw[red,thick] (b2)--(b4);

   \node(a) at (0,2.2) {$Q_1$};
   \node(a) at (0,-2.2) {$Q_2$};
      \node(a) at (0,-0.2) {$e$};

   \node(a) at (-3.3,1.5) {$X_1$};
   \node(a) at (3.3,1.5) {$X_3$};
    \node(a) at (-3.3,-1.5) {$X_2$};
     \node(a) at (3.3,-1.5) {$X_4$};
    \end{tikzpicture}
    \caption{The union of the blue curves depicts the path $Q_1$, and the union of the red curves depicts $Q_2$.}
   \label{fig:edgeweight}
\end{figure}


Next, suppose $N_1\neq N_2$. Let $M$ be a supernode different from $N_1, N_2$ and let $\{n_1,m_1\}$ be the unique edge between $N_1$ and $M$, and let $\{n_2,m_2\}$ be the unique edge between $N_2$ and $M$, where $m_i \in M$ and $n_i \in N_i$. Let $Q_1$ be the path between $m_1$ and $m_2$ in $M$.
 Let $P_1$ be the path in $N_1$ from $u$ to $n_1$, and let $P_2$ be a path (internally vertex-disjoint from $P_1$) in $N_1$ which starts at $u$, goes to a leaf in $G[N_1]$, and finishes at a vertex $s_1$ in another supernode $M_1$ different from $N_1,N_2,M$, such that $V(P_1) \cap M_1=\{s_1\}$.
Analogously, we find a path $P_3$ from $v$ to $n_2$ in $N_2$, and a path $P_4$ from $v$ to a vertex $s_2\in M_2$, for a supernode $M_2$ different from $N_1,N_2,M$, and such that $V(P_4) \cap M_2=\{s_2\}$. Finally, let $Q_2$ be the path in the tree $G[M_1\cup M_2]$ which connects $s_1$ to $s_2$.
Again we have constructed two cycles $C_1$ and $C_2$, consisting of paths $\{v,u\}-P_1-\{n_1,m_1\}-Q_1-\{m_2,n_2\}-P_3$ and $\{v,u\}-P_2-Q_2-P_4$, and as in the case when $N_1=N_2$, we get a permissible cycle $C$ (with edge-set $(E(C_1) \cup E(C_2))\setminus\{e\}$) contained in the union of at most five supernodes. We conclude again that $2w(e)=w(C_1)+w(C_2)-w(C)\in B$.

Now let us show that all permissible cycles $C$ have $w(C)\in B$, and for the sake of contradiction assume the contrary. Let $C$ be a permissible cycle with $w(C) \notin B$ that minimises the number of supernodes it intersects. By assumption, $C$ is incident to at least six supernodes.
Look at the cyclic ordering of those supernodes based on their appearance on $C$ (where at most one supernode appears two times, as shown in Figure \ref{fig:largepermissible}). 

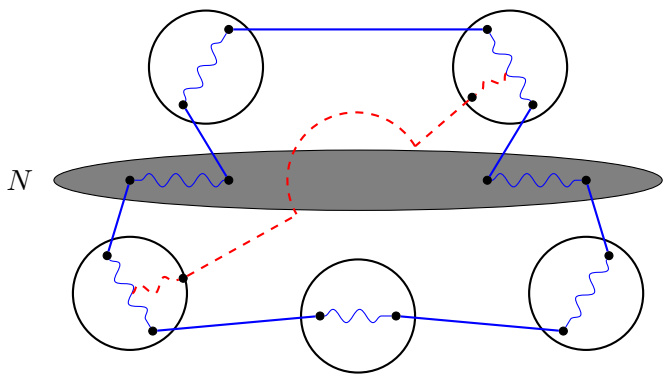
\begin{figure}[ht]
    \centering
    \begin{tikzpicture}[xscale=1,main node/.style={circle,draw,color=black,fill=black,inner sep=0pt,minimum width=3pt}]
        
	    \node(a) at (0,0) [label=left:$N$]{};
	    \node(b) at (8,0){};
	    \fitellipsis{a}{b};

	    \draw[thick](2,1.5) circle (0.75);
	        \node[main node] (a1) at (1.7,1){};
	        \node[main node] (b1) at (2.3,2){};
	        \path [draw=blue,snake it] (a1) -- (b1) ;
	    \draw[thick](6,1.5) circle (0.75);
	        \node[main node] (a2) at (6.3,1){};
	        \node[main node] (b2) at (5.7,2){};
	        \path [draw=blue,snake it] (a2) -- (b2) ;
	    \draw[thick](1,-1.5) circle (0.75);
	        \node[main node] (a6) at (0.7,-1){};
	        \node[main node] (b6) at (1.3,-2){};
	        \path [draw=blue,snake it] (a6) -- (b6) ;
	    \draw[thick](7,-1.5) circle (0.75);
	        \node[main node] (a4) at (7.3,-1){};
	        \node[main node] (b4) at (6.7,-2){};
	        \path [draw=blue,snake it] (a4) -- (b4) ;
	    \draw[thick](4,-1.8) circle (0.75);
	        \node[main node] (a5) at (3.5,-1.8){};
	        \node[main node] (b5) at (4.5,-1.8){};
	        \path [draw=blue,snake it] (a5) -- (b5) ;
	        
	        \node[main node] (a3) at (5.7,0){};
	        \node[main node] (b3) at (7,0){};
	        \path [draw=blue,snake it] (a3) -- (b3) ;
	        \node[main node] (a7) at (2.3,0){};
	        \node[main node] (b7) at (1,0){};
	        \path [draw=blue,snake it] (a7) -- (b7) ;

	         \draw[thick, blue] (b1)--(b2);
            \draw[thick, blue] (a2)--(a3);
            \draw[thick, blue] (b3)--(a4);
            \draw[thick, blue] (b4)--(b5);
            \draw[thick, blue] (a5)--(b6);
            \draw[thick, blue] (a6)--(b7);
            \draw[thick, blue] (a7)--(a1);

       \node[main node] (c) at (1.7,-1.3){};
       \path [draw=red,snake it, thick, dashed] (1.03,-1.5) -- (c) ;
       \path [draw=red, thick,dashed] (c) -- (3.21,-0.45);
       
       \node[main node] (d) at (5.5,1.1){};
       \path [draw=red,snake it, thick,dashed] (5.949,1.42) -- (d) ;
       \path [draw=red, thick,dashed] (d) -- (4.75,0.45);

   \draw[draw=red, thick,dashed] (4.75,0.45) arc (30:210:0.9);
	    
        
        
         
         

	    
	    

    \end{tikzpicture}
    \caption{The circles represent the supernodes that contain a single path in the cycle $C$ for which we assumed that $w(C)\notin B$, while the ellipse represents the supernode $N$ that intersects $C$ in two paths. The curved blue lines are paths in $C$, while straight blue lines are edges in $C$. The dashed red curve represents the path $P$ between two supernodes which splits $C$ into two permissble cycles, each of which is incident to fewer supernodes than $C$.}
   \label{fig:largepermissible}
\end{figure}

Let $N_1$ and $N_2$ be two supernodes that intersect $C$ in exactly one path each, $P_1$ and $P_2$ respectively, such that $N_1$ and $N_2$ are at least three apart in the cyclic ordering of supernodes. Now, let $P$ be the shortest path in $G[N_1\cup N_2]$ between $P_1$ and $P_2$. The union $C \cup P$ then splits into two cycles $C_1$ and $C_2$ whose intersection is the path $P$. But since $C_1$ and $C_2$ are also permissible, and since they are incident to fewer supernodes than $C$, by our minimality assumption on $C$ we have that $w(C_1),w(C_2)\in B$. But this means that $$w(C)=\underbrace{w(C_1)}_{\in B}+\underbrace{w(C_2)}_{\in B}-\underbrace{2w(P)}_{\in B}\in B,$$ since $2w(e)\in B$ for each edge $e$ in the path $P$, as shown in the first part of the proof. This gives the desired contradiction and finishes the proof.
\end{proof}

We continue with a simple lemma which, roughly speaking, will be used to reveal some information about the cycle weights in a graph in which a particular connector cannot be extended by another cycle. It is one ingredient for the proposition which follows after the lemma.
\begin{lemma}\label{lem:subgroup}
Let $S$ and $T$ be subsets of elements of a finite abelian group $A $, where $0\in S$. Suppose that $S+\{t\}\subseteq S$ for all $t\in T$. Then $\langle \{t\mid t\in T\}\rangle\subseteq S$.
\end{lemma}
\begin{proof}
We want to show that for every integer $k$ and for every sequence $t_1,\ldots,t_k$ of (not necessarily distinct) elements of $T$, we have $\sum_{i=1}^kt_i\in S$. We show this by induction on $k$.

For $k=1$, note that $S$ contains $0$, so by assumption $0+t_1\in S$. For $k \ge 2$, by the induction hypothesis we may assume $\sum_{i=1}^{k-1}t_i\in S$. Then we have $\sum_{i=1}^k t_i=\left(\sum_{i=1}^{k-1}{t_i}\right)+t_k \in S+\{t_k\} \subseteq S$, as desired.
\end{proof}

The following result is one of the main building blocks of our proof. It states that in a $B$-restricted minor we can either find a $B$-connector of small size, or we can delete a small number of supernodes and be left with a $B'$-restricted minor for a proper subgroup $B'<B$.
\begin{proposition}\label{prop:connectors}
Let $A$ be a finite abelian group, $B \le A$ a subgroup, and let $G$ be a reduced $B$-restricted $A$-weighted $K_f$-minor. Then at least one of the following two claims holds:
\begin{itemize}
    \item $G$ contains a $B$-connector $F$ which intersects at most $7|B|$ supernodes, such that for each endpoint $v$ of $F$ and the supernode $N$ containing $v$ we have $V(F)\cap N=\{v\}$. Furthermore, the base path of $F$ is permissible.
    \item There is a proper subgroup $B'<B$ and a reduced subminor $G'\preceq_A G$ such that $G'$ is a $B'$-restricted $A$-weighted reduced $K_{f'}$-minor, where $f'\geq f-7|B|$.
\end{itemize}
\end{proposition}
\begin{proof}
Note that we may assume $f>7|B|$, as otherwise the second claim is trivially true by letting $G'$ be a trivial subgraph of $G$. We will attempt to construct the connector by finding a sequence of subsets $S_0\subsetneq S_1\subsetneq\ldots\subsetneq S_t=B$ for some $t\leq |B|$, and for each $i\in[t]$ an $S_i$-connector intersecting at most $7i$ supernodes. For every $i \in [t-1]$, the $S_{i+1}$-connector will extend the previously constructed $S_i$-connector with vertices from at most seven new supernodes.

Let $S_0=\emptyset$, and for technical reasons, with slight abuse of notation, let the empty subgraph be our first connector.
Suppose for some $0 \le i \le t-1$ we have found an $S_i$-connector where $S_i\subsetneq B$, and let us find an $S_{i+1}$-connector. Consider the graph obtained by removing from $G$ all supernodes that have a vertex in the $S_i$-connector, and denote by $G_i$ a reduced minor of that graph. Note that $G_i$ is a $K_{f'}$-minor for $f'\geq f-7i>7|B|-7i>7$ and that $G_i \preceq_A G$.

Let $C$ be a permissible cycle in $G_i$ that is incident to at most five supernodes. 
Note that $C$ consists of vertices that are in at least three different supernodes (as any pair of supernodes induces a tree in a reduced minor),
 and choose $T_1,T_2$ and $T_3$ to be three distinct supernodes that intersect $C$ in precisely one non-empty path (see Remark~\ref{rem:permissible}). Now, let $N_1,N_2$ and $N_3$ be three supernodes in $G_i$ disjoint from $C$.

 In what follows, to simplify our notation, we will use the same names for corresponding supernodes in $G$ and in the reduced subminor $G_i$, anticipating that a supernode might lose some vertices and that some edges can be contracted when passing from $G$ to $G_i$. We also use the same names for subgraphs $H$ of $G_i$, which correspond to subdivisions of $H$ in $G$. Observe that by Remark~\ref{rem:minimal}, the weight of the subdivision paths in $G$ is the weight of the corresponding edge of $H$ in $G_i$.

For $j \in [3]$, let $u_j \in V(G_i)$ be the vertex in $N_j$ that has a neighbour in $T_j$, and let $Q_j$ be the shortest path from $u_j$ to $C \cap T_j$ in $G_i[N_j \cup T_j]$, observing that the interval vertices of $Q_j$ all lie in $T_j$ (see Figure \ref{fig:addingcycle}). The endpoints of these paths split $C$ into three paths in $G_i$, whose weights we denote by $x_1$, $x_2$ and $x_3$, in such a way that for every $j \in [3]$ the segment of $C$ between the endpoints of $Q_{j-1}$ and $Q_{j+1}$ is of total weight $x_j$\footnote{summation of indices with modular arithmetic}. Let $\delta_1(C)=x_1+x_2-x_3$, $\delta_2(C)=x_2+x_3-x_1$ and $\delta_3(C)=x_3+x_1-x_2$. 

\begin{figure}[h!]
    \centering
    \begin{tikzpicture}[xscale=1,main node/.style={circle,draw,color=black,fill=black,inner sep=0pt,minimum width=3pt}]
        
	    \node(a) at (0,0) [label=left:$N$]{};
	    \node(b) at (8,0){};
	    \fitellipsis{a}{b};

	    inside circles
	    \draw[thick](2,1.5) circle (0.75);
	        \node[main node] (a1) at (1.7,1){};
	        \node[main node] (b1) at (2.3,2){};
	        \path [draw=blue,snake it] (a1) -- (b1) ;
	    \draw[thick](6,1.5) circle (0.75);
	        \node[main node] (a2) at (6.3,1){};
	        \node[main node] (b2) at (5.7,2){};
	        \path [draw=blue,snake it] (a2) -- (b2) ;
	    \draw[thick](2,-1.5) circle (0.75);
	        \node[main node] (a6) at (1.7,-1){};
	        \node[main node] (b6) at (2.3,-2){};
	        \path [draw=blue,snake it] (a6) -- (b6) ;
	    \draw[thick](6,-1.5) circle (0.75);
	        \node[main node] (a4) at (6.3,-1){};
	        \node[main node] (b4) at (5.7,-2){};
	        \path [draw=blue,snake it] (a4) -- (b4) ;

	        \node[main node] (a3) at (5.7,0){};
	        \node[main node] (b3) at (7,0){};
	        \path [draw=blue,snake it] (a3) -- (b3) ;
	        \node[main node] (a7) at (2.3,0){};
	        \node[main node] (b7) at (1,0){};
	        \path [draw=blue,snake it] (a7) -- (b7) ;

	        \draw[thick, blue] (b1)--(b2);
            \draw[thick, blue] (a2)--(a3);
            \draw[thick, blue] (b3)--(a4);
            \draw[thick, blue] (b4)--(b6);
    
            \draw[thick, blue] (a6)--(b7);
            \draw[thick, blue] (a7)--(a1);

    	    \node(a) at (0,2.5) {$N_1$};
            \draw[thick](0,1.5) circle (0.75);
	        \node[main node] (c1) at (0.7,1.5){};
	         \node(a) at (0.4,1.5) {$u_1$};
	        
	        \node[main node] (d1) at (1.3,1.5){};
	        \path [draw=red,thick] (c1) -- (d1) ;
	        \path [draw=red,snake it] (d1) -- (1.9,1.5) ;
	        
	       	    \node(a) at (8,2.5) {$N_2$};
	       	    \node(a) at (7.6,1.5) {$u_2$};
	        \draw[thick](8,1.5) circle (0.75);
	        \node[main node] (c1) at (7.3,1.5){};
	        \node[main node] (d1) at (6.7,1.5){};
	        \path [draw=red,thick] (c1) -- (d1) ;
	        \path [draw=red,snake it] (d1) -- (5.96,1.5);
       
           	    \node(a) at (0,-2.5) {$N_3$};
           	    \node(a) at (0.4,-1.5) {$u_3$};
            \draw[thick](0,-1.5) circle (0.75);
	        \node[main node] (c1) at (0.7,-1.5){};
	        \node[main node] (d1) at (1.3,-1.5){};
	        \path [draw=red,thick] (c1) -- (d1) ;
	        \path [draw=red,snake it] (d1) -- (2.02,-1.5) ;
	    
 \node(a) at (2,2.5) {$T_1$};
  \node(a) at (6,2.5) {$T_2$};
   \node(a) at (2,-2.5) {$T_3$};
   
    \end{tikzpicture}
    \caption{The circles represent the supernodes that contain a single path in the cycle $C$, while the ellipse represents the supernode $N$ that intersects $C$ in two paths. The red curves from $u_j$ to the blue paths represent the paths $Q_j$.}
   \label{fig:addingcycle}
\end{figure}
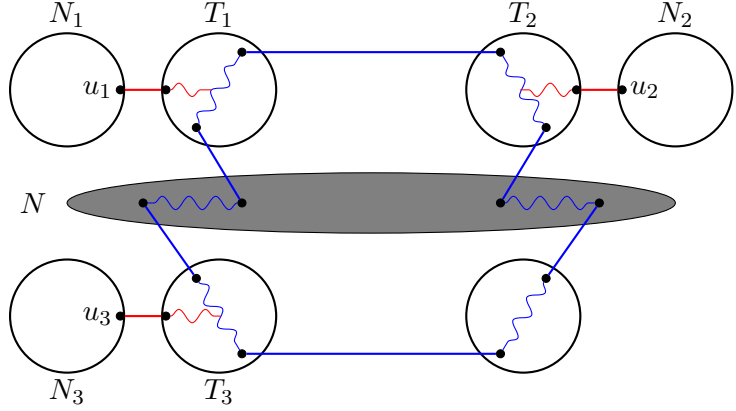

Suppose first that there is some $j \in [3]$ such that $S_i + \{0,\delta_j(C)\} \neq S_i$; by relabelling, we may assume that $j = 1$. Then the endpoints of $Q_1$ and $Q_2$ split $C$ into two paths $P_1$ and $P_2$ of weights $x_3$ and $x_1+x_2$. Now we complete the construction of the $S_{i+1}$-connector, where we set $S_{i+1}=S_i+\{0,\delta_1(C)\}$. 
If $i=0$, then let the $S_{i+1}$-connector be the concatenation of $Q_1,C$ and $Q_2$. Otherwise, let $z$ be the last vertex of the $S_i$-connector that we want to extend, and let $Z$ be the supernode in $G$ containing it. Prolong the last path of this connector by attaching to $z$ the shortest path from $z$ to $Q_1$ in $G[Z\cup N_1]$, and then attach to this path the remaining part of $Q_1$ whose endpoint is on the cycle $C$. Now, attach to the connector the cycle $C$. To complete the $S_{i+1}$-connector, we wish to attach the path $Q_2$. Note that the last vertex $u_2$ of $Q_2$ was the only vertex in $Q_2$ from $N_2$ in the graph $G_i$, but when we lift from $G_i$ to $G$, the path $Q_2$ may have more vertices from $N_2$. Hence we remove all vertices from $Q_2 \cap N_2$ except the one adjacent to $T_2$, and attach the obtained path to $C$. In this fashion, we maintain the desired property of the last vertex in the connector, as the vertex in $Q_2$ adjacent to $T_2$ is now the only vertex from $N_2$ in the connector. Also note that that we added vertices from at most seven new supernodes to the $S_i$-connector to make this $S_{i+1}$-connector. Finally, $P_1$ is a permissible path,\footnote{Either all supernodes incident to $C$ intersect $C$ in at most one path, and then it is clear that the same holds for $P_1$. Otherwise, by Remark \ref{rem:permissible} we know that exactly 5 supernodes are incident to $C$, and the cyclic ordering of the supernodes as they appear on $C$ is $T,-,-,T,-,-$ (where $T$ is the repeated supernode, and each '-' sign is uniquely assigned to a distinct supernode). Now it is clear that for every choice of $P_1$ as in the proof, $P_1$ intersects $G[T]$ in at most one path.} and by construction it is easy to see that the base path of the $S_{i+1}$-connector is also permissible.

If we are able to iterate this process, after at most $|B|$ steps we would obtain the desired $B$-connector (meeting at most $7|B|$ supernodes). On the other hand, if an iteration fails, it must mean that for some $0 \le i \le |B|-1$ and for all $j \in [3]$, we have $S_i + \{0,\delta_j(C)\} = S_i$ for every choice of a permissible cycle $C$ incident to at most five supernodes. Let $\mathcal{C}$ be the collection of those cycles in $G_i$. We conclude, by Lemma~\ref{lem:subgroup}, that the group $B'=\langle\{\delta_j(C)\mid j\in [3], C\in \mathcal{C}\}\rangle$, is contained in $S_i$. Since we do not yet have a $B$-connector, we must have $B' < B$.
 Furthermore, observe that $w(C)=\delta_1(C)+\delta_2(C)+\delta_3(C)\in B'$ for all $C\in \mathcal{C}$. Appealing to Lemma~\ref{lem:even paths}, we deduce that $G_i$ is a $B'$-restricted reduced $K_{f'}$-minor, where, since we have only lost the supernodes in the $S_i$-connector, $f'\geq f-7i\geq f-7|B|$.
\end{proof}

\section{The proof}\label{sec:proof}
We are now ready to prove our main result, which we restate here for the convenience of the reader. 
\abelian*
\begin{proof}
Let $G$ be an $A$-weighted $K_f$-minor, which we may assume is reduced (see Remark~\ref{rem:minimal}). The idea of the proof is to find $m = e(H)$ disjoint connectors in $G$, each of which is contained in the union of at most $7|A|$ supernodes. We will route the subdivision paths through these connectors, so that we can then apply switches within the connectors to ensure each path is of weight $0 \in A$.

We begin by constructing $A$-connectors for as long as we are able. If at some point no further such connector can be found, this will reveal some structural information about the edge weights in the remaining graph, as given by Proposition~\ref{prop:connectors}. This information will allow us to pass to some subgroup $A' < A$, and we shall then construct $A'$-connectors instead. We repeat this process until we have the desired number of connectors, keeping track of their various types. We shall then use these in conjunction with the remaining vertices to build an $A$-divisible subdivision of $H$.

\medskip
 
We initially set $f_0=f$, $B_0=A$ and $G_0=G$. For each iteration of our process, suppose we have $i \ge 0$ and $f_i$, $B_i$ and $G_i$ such that $G_i \preceq_A G$ is a $B_i$-restricted $A$-weighted reduced $K_{f_i}$-minor; we shall then find $f_{i+1}$, $B_{i+1}$ and $G_{i+1}$ as follows.

Applying Proposition~\ref{prop:connectors} to $G_i$, we either find a $B_i$-connector $F_i$ using at most $7|B_i|$ supernodes, or we obtain a proper subgroup $B' < B_i$ and a reduced subminor $G'\preceq_A G_i \preceq_A G$ that is a $B'$-restricted reduced $K_{f'}$-minor, where $f'\geq f-7|B_i|$. 
 \begin{itemize}
     \item[(1)] 
 In the former case, we increment our iteration counter, and let $G_{i+1}$ be the graph obtained by reducing $G_i$ after removing the supernodes used in $F_i$. We set $B_{i+1}=B_i$ and we let $f_{i+1}$ be the number of supernodes in $G_{i+1}$, noting that $f_{i+1}\geq f_i-7|B_i|$. 
 \item[(2)] In the latter case, we remain in the current iteration, but update $G_{i}=G'$, $B_{i}=B'$ and we set $f_i=f'$. We repeat the process, applying Proposition~\ref{prop:connectors} to $G_i$ again, until we encounter the first case.
 \end{itemize}
 
Every time we enter case (2) in this inductive process, the size of the group $B_i$ we are working with decreases by a factor of at least two. Thus, we encounter case (2) at most $\lceil \log_2 |A| \rceil$ times. Moreover, after $j \ge 0$ instances of case (2), the size of the group $B_i$ is at most $2^{-j} |A|$. Since we lose at most $7|B_i|$ supernodes every time we fall into case (2), we can bound the total number of supernodes lost in case (2) during the process by $14|A| = \sum_{j \ge 0} 7 \cdot 2^{-j} |A|$.

As the number of occurrences of the second case is bounded, we must eventually meet the first case $m$ times, after which we will have the desired $m$ connectors. Each connector costs us at most $7|B_i| \le 7|A|$ supernodes, and so by the time we reach the final graph $G_m$, we still have at least $f - m \cdot 7|A| - 14|A| \ge 4 n \sigma(A)$ supernodes, and for each $0 \le i < m$ we have a $B_i$-connector $F_i \subseteq G_i$.

Now pick $4n\sigma(A)$ of the supernodes in $G_m$ and group them into clusters of four.
For each such cluster of supernodes $N_1,N_2,N_3$ and $N_4$, let $n_2,n_3$ and $n_4$ be the vertices in $N_2,N_3$ and $N_4$ respectively that have a neighbour in $N_1$. Mark the central vertex in $G_m[N_1]$ for the neighbours of $n_2,n_3$ and $n_4$ in $N_1$.
Let $\{x_i\mid i\in [n\sigma(A)]\}$ be the collection of the marked central vertices, and denote by $X_i$ the supernode of $x_i$ for each $i$, and we call $X_i$ the \emph{central} supernode of its cluster.

Next we show that we can pick $n$ of those central vertices $$\{y_i\mid i\in[n]\}\subseteq \{x_i\mid i\in [n\sigma(A)]\},$$ where we denote the supernodes of $y_i$ by $Y_i$, such that, for every pair $i,j \in [n]$, the weight of the unique path between $y_i$ and $y_j$ in the tree $G_m[Y_i\cup Y_j]$ is in $B_m$.

Indeed, let $B':=\{a\in A\mid 2a\in B_m\}$. Since $G_m$ is $B_m$-restricted, we have that every edge $e$ in $G_m$ satisfies $w(e)\in B'$. Noting that $B_m$ is, trivially, a subgroup of $B'$, we can define the quotient group $B^*=B'/B_m$ and, by definition of $\sigma$, we have $|B^*|=\frac{|B'|}{|B_m|}\leq \sigma(A)$.

We now colour the vertices $x_i$ with the elements of $B^*$. We start by giving $x_1 \in X_1$ the colour $0 + B_m \in B^*$. Then, for each $i \ge 2$, let $P_{x_i}$ be the unique path between $x_1$ and $x_i$ in $G[X_1 \cup X_i]$, and colour the vertex $x_i$ with $w(P_{x_i})+B_m \in B^*$.

We choose the vertices $y_i$ arbitrarily from the largest colour class, which is of size at least $\frac{n\sigma(A)}{|B^*|}  \ge\frac{n\sigma(A)}{\sigma(A)}=n$. Let us show that the path $P_{ij}$ between $y_i$ and $y_j$ in $G_m[Y_i\cup Y_j]$ is of weight in $B_m$ for all pairs $i,j\in[n]$. Let $P_i=P_{y_i}, P_j=P_{y_j}$, $Q_1=P_i\cap P_{ij}$, $Q_2=P_j\cap P_{ij}$, and $Q_3=P_i\cap P_j$, and let the respective weights of the latter three paths\footnote{Note that the intersections $Q_1,Q_2,Q_3$ indeed form paths (possibly of length zero), since $G_m[Y_i],G_m[Y_j]$ and $G_m[X_1]$ are trees.} be $q_1$, $q_2$ and $q_3$ respectively (see Figure~\ref{fig:colouringworks}). Let $C$ be the permissible cycle formed by the edges in $E[P_i\cup P_j\cup P_{ij}]\setminus E[Q_1\cup Q_2\cup Q_3]$. Observe that $w(P_{ij})=w(C)+2(q_1+q_2+q_3)-(w(P_i)+w(P_j))$. Since $G_m$ is $B_m$-restricted and since we chose $y_i$ and $y_j$ to be of the same colour, we have $w(P_i)+w(P_j)\in 2w(P_i)+B_m = B_m$. Furthermore, since $C$ is permissible in $G_m$, we have $w(C)\in B_m$. Finally $2(q_1+q_2+q_3)\in B_m$, again from $G_m$ being $B_m$-restricted. This yields $w(P_{ij})\in B_m$, as we wanted.

\begin{figure}
    \centering
    \begin{tikzpicture}[scale=1,main node/.style={circle,draw,color=black,fill=black,inner sep=0pt,minimum width=3pt}]
        

	    \draw[thick](-1.5,0) circle (1);
	        \node[main node] (a1) at (-2,-0.5){};
	         \node(a)[scale=0.8] at (-2,-0.7) {$y_i$};
	          \node(a)[scale=0.8] at (-1.9,-0.1) {$Q_1$};
	        \node[main node] (a2) at (-1.5,0){};
	        \node[main node] (a3) at (-1.3,0.7){};
	        \node[main node] (a4) at (-0.9,0.1){};
	        
	        \path [draw=black,snake it] (a1) -- (a2) ;
	        \path [draw=black,snake it] (a2) -- (a3) ;
	        \path [draw=black,snake it] (a4) -- (a2) ;
	        
	    \draw[thick](1.5,0) circle (1);
	      \node[main node] (b1) at (2,-0.5){};
	      \node(a)[scale=0.8] at (1.93,-0.7) {$y_j$};
	       \node(a)[scale=0.8] at (2,-0.1) {$Q_2$};
	        \node[main node] (b2) at (1.5,0){};
	        \node[main node] (b3) at (1.3,0.7){};
	        \node[main node] (b4) at (0.9,0.1){};
	        
	        \path [draw=black,snake it] (b1) -- (b2) ;
	        \path [draw=black,snake it] (b2) -- (b3) ;
	        \path [draw=black,snake it] (b4) -- (b2) ;
	     
	    \draw[thick](0,2.5) circle (1);
	        \node[main node] (c1) at (0,3.2){};
	        \node(a)[scale=0.8] at (0,3.35) {$x_1$};
	         \node(a)[scale=0.8] at (0.26,2.8) {$Q_3$};
	        \node[main node] (c2) at (0,2.5){};
	        \node[main node] (c3) at (-0.5,1.8){};
	        \node[main node] (c4) at (0.5,1.8){};
	        
	        \path [draw=black,snake it] (c1) -- (c2) ;
	        \path [draw=black,snake it] (c2) -- (c3) ;
	        \path [draw=black,snake it] (c4) -- (c2) ;
	     
	       \path [draw=black,thick,dotted] (a3) -- (c3) ;
	        \path [draw=black,thick,dotted] (a4) -- (b4) ;
	         \path [draw=black,thick,dotted] (c4) -- (b3) ;
	 
\node(a) at (-2.8,0) {$Y_i$};
\node(a) at (2.8,0) {$Y_j$};
\node(a) at (1.1,3.2) {$X_1$};

   
    \end{tikzpicture}
    \caption{The dotted lines represent edges between two supernodes, while the solid curves represent paths in $G_m$}
   \label{fig:colouringworks}
\end{figure}
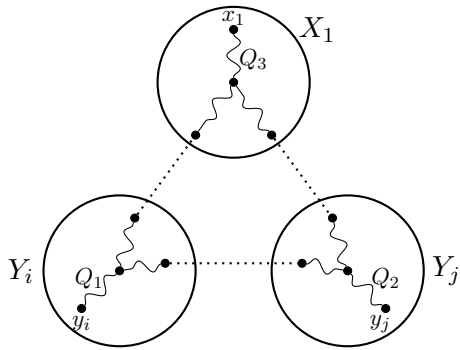

We now set about building an $A$-divisible subdivision of $H$. If we denote by $\{v_i \mid i \in [n]\}$ the vertices of $H$, $y_i$ will be the branch vertex corresponding to $v_i$. Let $\mathcal{V} = \{ y_i \mid i \in [n] \}$. For a fixed arbitrary ordering $e_1,\ldots, e_m$ of the edges of $H$, we now show how to construct the subdivision paths between the corresponding branch vertices in $\mathcal{V}$.

For each $k \in [m]$, the subdivision path of weight $0\in A$ corresponding to the edge $e_k$ will be constructed within the graph $G_k$. Since passing to a reduced subminor does not change the weights of paths, the corresponding path in $G$ will also be of weight $0$. After describing these paths, we will explain why they are internally vertex-disjoint.

Importantly, as an invariant during the construction process we will require that for every edge $e_k$ with $k \in [m]$ and with endpoints $v_i$ and $v_j$, the subdivision path representing $e_k$ is vertex-disjoint from all supernodes contained in clusters corresponding to vertices $y_t$ with $t \in [n]\setminus \{i,j\}$, and that it intersects the clusters of $y_i$ and $y_j$ only in their central supernodes as well as at most one other supernode from the cluster.

Now suppose that for some $k \in [m]$ we have already constructed subdivision paths for $e_1,\ldots,e_{k-1}$, and let $e_k=\{v_i,v_j\}$. We thus need to build a subdivision path connecting $y_i$ and $y_j$. Recall that $y_i$ came from the supernode $Y_i$, which was part of a four-supernode cluster. Let $Y_i^1$, $Y_i^2$ and $Y_i^3$ be the other three supernodes from that cluster, and let $y_i^1, y_i^2$ and $y_i^3$ be the vertices in these supernodes that are adjacent to $Y_i$ (for which $y_i$ was the central vertex). We define $y_j^1 \in Y_j^1, y_j^2 \in Y_j^2$ and $y_j^3 \in Y_j^3$ similarly with respect to $y_j$.

By our invariant on how subdivision-paths interact with the clusters, at most two of the three non-central supernodes $Y_i^1, Y_i^2, Y_i^3$ respectively $Y_j^1, Y_j^2, Y_j^3$ are intersected by subdivision-paths corresponding to $e_1,\ldots,e_{k-1}$. Hence, without loss of generality we may assume that $Y_i^1$ and $Y_j^1$ have not been used previously in the construction of subdivision paths. Let $P$ be the path in $G_m[Y_i \cup Y_j]$ between $y_i$ and $y_j$, $Q^k_1$ the path in $G_m[Y_i \cup Y_i^1]$ between $y_i$ and $y_i^1$, and $Q^k_2$ the path in $G_m[Y_j \cup Y_j^1]$ between $y_j$ and $y_j^1$. We then lift these paths to the corresponding paths in $G_k$; that is, we reverse any contractions that may have occurred when passing from $G_k$ to $G_m$.

First, recall that our choice of $y_i$ and $y_j$ guarantees $w_{G_m}(P) \in B_m$, and so we have $w_{G_k}(P) \in B_m \le B_k$. Next, note that the paths $Q^k_1$ and $Q^k_2$ may gain some additional vertices in $G_k$. We define $p_1 \in Y_i^1$ to be the vertex on $Q^k_1$ that is adjacent to $Y_i$, and let $\bar{Q}^k_1 \subseteq Q^k_1$ be the path in $G_k[Y_i \cup Y_i^1]$ from $y_i$ to $p_1$. We define $p_2$ and $\bar{Q}^k_2$ analogously.

Now let $t_1$ and $t_2$ be the endpoints of the $k^\text{th}$ connector $F_k$, which is a $B_k$-connector in $G_k$, and let $T_1$ and $T_2$ be their corresponding supernodes. Let $\bar{Q}^k_3$ be the path in $G_k[Y_i^1\cup T_1]$ connecting $p_1$ to $t_1$, and let $\bar{Q}^k_4$ be the path in $G_k[Y_j^1\cup T_2]$ between $p_2$ and $t_2$ (see Figure \ref{fig:coloringworks}). Finally, let $\bar{Q}^k_5$ be the base path between $t_1$ and $t_2$ in the connector $F_k$. Observe that the concatenation of ${P},\bar{Q}^k_1,\bar{Q}^k_3,\bar{Q}^k_5,\bar{Q}^k_4$ and $\bar{Q}^k_2$ gives a permissible cycle in $G_k$, whose weight is hence in $B_k$. Indeed, the base path $\bar{Q}^k_5$ is permissible by construction, while the other paths are contained within two supernodes, and hence must be permissible. As the paths do not share any supernodes beyond those of their common endpoints, it is then easy to verify that their union is permissible (with each supernode intersecting the cycle in at most one path).

Therefore, by removing the edges of $P$ from the cycle (or, equivalently, by concatenating $\bar{Q}^k_1,\bar{Q}^k_3,\bar{Q}^k_5,\bar{Q}^k_4$ and $\bar{Q}^k_2$), we obtain a path $\bar{Q}^k$ in $G_k$ between $y_i$ and $y_j$ whose weight is in $B_k$. We can thus perform the appropriate switches in the $B_k$-connector $F_k$ (thereby modifying the path $\bar{Q}^k_5$) to find a path in $G_k$ between $y_i$ and $y_j$ of weight $0\in A$. Finally, this lifts to a path $Q_{ij}$ between $y_i$ and $y_j$ in the original graph $G$ of weight $0\in A$, which is what we sought.

\begin{figure}[ht]
    \centering
    \begin{tikzpicture}[scale=0.7,main node/.style={circle,draw,color=black,fill=black,inner sep=0pt,minimum width=3pt}]
  
	    \draw[thick](0,0.75) circle (0.5);
	    \draw[thick](0,2.25) circle (0.5);
	    
	    \draw[thick](0,-2.25) circle (0.5);
	    
	    \path [draw=black,thick] (0,-2.75) -- (0,-3.25) ;
	    \path [draw=black,thick] (0,2.75) -- (0,3.25) ;
	    \path [draw=black,thick] (0,1.25) -- (0,1.75) ;
	    \path [draw=black,thick] (0,0.25) -- (0,0 ) ;
	    \path [draw=black,thick] (0,-1.5) -- (0,-1.75) ;
	    \draw[loosely dotted,thick] (-0,-0.1)--(0,-1.4);
	    
	        \node[main node] (a1) at (0,-3.25){};
	        \node(a)[scale=1] at (0,-3.6) {$t_1$};
	  
	        \node[main node] (a2) at (0,3.25){};
	        \node(a)[scale=1] at (0,3.6) {$t_2$};

	        \draw[thick,dotted](2,3) circle (1);
	         \draw[thick,dotted](4.5,2.5) circle (1);
	         
	         \draw[thick,dotted](2,-3) circle (1);
	         \draw[thick,dotted](4.5,-2.5) circle (1);
	   
	    \node[main node] (c1) at (2,3){};
	    \node(a)[scale=1] at (2,3.5) {$y_j^1$};
	    \node(a)[scale=1] at (2,1.5) {$Y_j^1$};
	    
	    \node[main node] (p2) at (2.8,2.84){};
	    \node(a)[scale=1] at (2.7,2.5) {$p_2$};
	    
	    \node[main node] (c2) at (4.5,2.5){};
	    \node(a)[scale=1] at (4.5,3) {$y_j$};
	     \node(a)[scale=1] at (5.8,3) {$Y_j$};
	    
	    \node[main node] (d1) at (2,-3){};
	    \node(a)[scale=1] at (2,-3.5) {$y_i^1$};
	    \node(a)[scale=1] at (2,-1.5) {$Y_i^1$};
	    
	    \node[main node] (p1) at (2.8,-2.84){};
	    \node(a)[scale=1] at (2.7,-2.5) {$p_1$};
	    
	    \node[main node] (d2) at (4.5,-2.5){};
	    \node(a)[scale=1] at (4.5,-3) {$y_i$};
	      \node(a)[scale=1] at (5.8,-3) {$Y_i$};
	    
	    \draw [decorate,decoration={snake,amplitude=.4mm,segment length=2mm,post length=0.1mm}] (c1)--(c2) ;
	    \draw [decorate,decoration={snake,amplitude=.4mm,segment length=2mm,post length=0.1mm}] (a2) to [out = 330, in = 180] (p2) ;
	  
	    \draw [decorate,decoration={snake,amplitude=.4mm,segment length=2mm,post length=1mm}]
	     (d1)--(d2);
	    \draw [decorate,decoration={snake,amplitude=.4mm,segment length=2mm,post length=1mm}]
	     (a1) to [out = 30, in = 180] (p1);

    \end{tikzpicture}
    \caption{This figure illustrates how the path connecting $y_i$ and $y_j$ is constructed. Dotted circles represent supernodes, while squiggly lines represent paths.}
    \label{fig:coloringworks}
\end{figure}
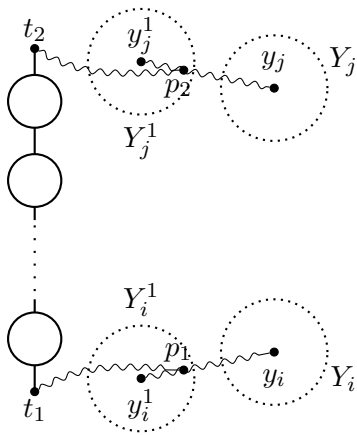

To finish, we show that the subdivision paths we have constructed are internally vertex-disjoint. Indeed, with the exception of the supernodes $\{Y_i \mid i \in [n] \}$ housing the branch vertices $\{ y_i \mid i \in [n] \}$, the subdivision paths pass through disjoint sets of supernodes. Within the supernode $Y_i$, since $y_i$ was the central vertex for $y_i^1$, $y_i^2$ and $y_i^3$, it follows that the subdivision paths are disjoint apart from $y_i$, as required. Thus, the subdivision paths are internally vertex-disjoint, and we have constructed a genuine $A$-divisible subdivision of $H$.
\end{proof}

\section{Concluding remarks} \label{sec:conclusion}

In this paper, we have addressed a problem of Alon and Krivelevich~\cite{alon} on divisible subdivisions, showing that if $H$ is a subcubic graph on $n$ vertices with $m$ edges, $A$ is a finite abelian group and $f \ge 7m|A| + 4n \sigma(A) + 14|A|$, then every $A$-weighted $K_f$-minor contains an $A$-divisible $H$-subdivision. In particular, by taking $A = (\mathbb{Z}_q, +)$, it follows that having $f \ge 7mq + 8n + 14q$ suffices to ensure the existence of an $H$-subdivision whose subdivision paths are all of length divisible by $q$.

This bound is tight up to a multiplicative constant. Indeed, since each subdivision path must have length at least $q$, a $q$-divisible $H$-subdivision requires at least $m(q-1) + n$ vertices. Thus, for $f = m(q-1) + n - 1$, $K_f$ itself is a $K_f$-minor without any $q$-divisible $H$-subdivisions. Having failed to find any constructions that yield a better lower bound, we suspect that this trivial bound may in fact be the true answer. Part of the difficulty in proving this is that $K_f$-minors can have a wide range of structures, and so it is easier to restrict the class of host graphs under consideration. A natural first step would be to only look at subdivisions of $K_f$.

\begin{problem}
Given $q \in \mathbb{N}$ and a subcubic graph $H$ with $n$ vertices and $m$ edges, is it true for $f = m(q-1) + n$ that every subdivision of $K_f$ contains a $q$-divisible $H$-subdivision?
\end{problem}

\noindent Through some case analysis, we could answer this in the affirmative when $q=2$ and $H = K_4$, providing some scant evidence in support of a positive answer. It would be interesting to see a general argument that applies to all $q$ and $H$.
\medskip

Recall that we seek $q$-divisible subdivisions because we cannot be guaranteed to find anything else --- there are $K_f$-minors where every path between vertices of degree at least three has length divisible by $q$. Our proof shows that this is essentially the only obstruction, since the $\mathbb{Z}_q$-connectors allow us to obtain paths of any parity we wish. Thus, given $q \ge 2$, a subcubic graph $H$ and, for every edge $e \in E(H)$, a residue $r_e \in \mathbb{Z}_q$, then in any $K_f$-minor $G$ we can either find an $H$-subdivision such that, for each edge $e$, the subdivision path corresponding to $e$ has length $r_e$ modulo $q$, or we find a proper subgroup $W < \mathbb{Z}_q$ and a subgraph $G' \subseteq G$ that is a $K_{f'}$-minor for some $f' \ge f - 7qm$, such that every path $P$ in $G'$ between vertices of degree at least three has length satisfying $2 \ell(P) \in W$. In particular, if $q$ is prime, then it divides $\ell(P)$.

\medskip

The other restriction we imposed is that the graph $H$ should be subcubic. This is again necessary, as there are $K_f$-minors with maximum degree three. However, we can circumvent this obstruction by including a large minimum degree requirement. Thomassen~\cite{thomassen} proved that for any graph $H$, all graphs of sufficiently large minimum degree contain $q$-divisible $H$-subdivisions (in fact, one can impose a much wider range of modular restrictions on the path lengths). In light of our results, it is natural to ask how much the bound on the minimum degree can be reduced when the host graph is a $K_f$-minor.

\begin{problem}
Given $f,q \ge 2$ and a graph $H$, what is the smallest $d = d(H,q,f)$ such that every $K_f$-minor $G$ with $\delta(G) \ge d$ contains a $q$-divisible $H$-subdivision?
\end{problem}

Finally, returning to the setting of group-weighted graphs, we observe that while our bound gives the correct order of magnitude in the case $A = \mathbb{Z}_q$, there is scope for improvement for other abelian groups. Indeed, when $H$ is a cycle, Scheucher, Sidorenko and the third author (personal communication) show that when $A = \mathbb{Z}_2^d$, $f$ need only grow logarithmically with the size of the group. It is then natural to presume that such savings can also be made for other subcubic graphs.

\begin{problem}
Given $d \ge 2$ and a subcubic graph $H$ on $n$ vertices with $m = \Omega(n)$ edges, is there some $f = O(dm)$ such that every $\mathbb{Z}_2^d$-weighted $K_f$-minor contains a $\mathbb{Z}_2^d$-divisible $H$-subdivision?
\end{problem}

\end{document}